\documentclass{amsart}
\usepackage{amsmath}
\usepackage{amsthm}
\usepackage[curve]{xypic}
\usepackage{dsfont}
\usepackage{latexsym}
\usepackage[T1]{fontenc}
\usepackage[utf8]{inputenc}
\usepackage{setspace}
\usepackage{varioref}
\usepackage{overpic}
\usepackage{verbatim}

\newtheorem{lemma}{Lemma}[section]
\newtheorem{sublemma}[lemma]{Sublemma}

\newtheorem{theorem}[lemma]{Theorem}
\newtheorem{corollary}[lemma]{Corollary}

\theoremstyle{definition}
\newtheorem{df}[lemma]{Definition} 

\def\id{\mathrm{id}}

\def\IN{\mathds{N}}
\def\IR{\mathds{R}}

\def\Inv{\mathrm{Inv}}

\def\eps{\varepsilon}

\newcommand\spr[1]{\mathord{\langle#1\rangle}}

\newcommand\abs[1]{\mathord{\left\lvert#1\right\rvert}}
\newcommand\norm[1]{\mathord{\left\Vert#1\right\Vert}}

\newcommand\loc{\mathrm{loc}}

\DeclareMathOperator{\codim}{codim}

\DeclareMathOperator{\Null}{\mathcal{N}}
\DeclareMathOperator{\Ran}{\mathcal{R}}
\DeclareMathOperator{\cl}{cl}
\DeclareMathOperator{\spn}{span}

\DeclareMathOperator{\Def}{\mathcal{D}}

\DeclareMathOperator{\grad}{\nabla}
\newcommand{\de}{\mathrm{d}}

\numberwithin{equation}{section}
\numberwithin{figure}{section}

\title[Gradient-like asymptotically autonomous equations]{The generic gradient-like structure of certain asymptotically autonomous semilinear parabolic equations}
\author{A. J\"anig}
\email{axel.jaenig@uni-rostock.de}
\address{Institut f\"ur Mathematik, Universit\"at Rostock, 18051 Rostock, Germany}

\begin{document}

\begin{abstract}
  We consider asymptotically autonomous semilinear parabolic equations
  \begin{equation*}
    u_t + Au = f(t,u).
  \end{equation*}
  Suppose that $f(t,.)\to f^\pm$ as $t\to\pm\infty$, where the semiflows
  induced by
  \begin{equation}
    \label{eq:140602-1511}
    u_t + Au = f^\pm(u) \tag{*}
  \end{equation}
  are gradient-like. Under certain assumptions,
  it is shown that generically with respect to a perturbation $g$
  with $g(t)\to 0$ as $\abs{t}\to\infty$, every solution of
  \begin{equation*}
    u_t + Au = f(t,u) + g(t)
  \end{equation*}
  is a connection between equilibria $e^\pm$ of \eqref{eq:140602-1511}
  with $m(e^-)\geq m(e^+)$. Moreover, if the Morse indices satisfy
  $m(e^-) = m(e^+)$, then
  $u$ is isolated by linearization.
\end{abstract}

\maketitle

\newcommand{\X}{\mathcal{X}}
\newcommand{\Y}{\mathcal{Y}}
\newcommand{\Z}{\mathcal{Z}}
\begin{section}{Introduction}
  Let $\Omega\subset \IR^m$, $m\geq 1$ be a bounded domain with
  smooth boundary. As an illustrative example for the abstract result
  in the following section, consider the following problem
  \begin{align}
    \label{eq:dirichlet}
            \partial_t u - \Delta u  &= f(t,x,u(t,x),\nabla u(t,x)) \\
            \notag  u(t,x) & = 0 & x\in\partial\Omega\\
            \notag  u(t,x) & = u_0(x) & x\in\Omega
  \end{align}

  Suppose that $f$ is sufficiently regular and $f(t,x,u,v)\to f^\pm(x,u)$ as $t\to\pm\infty$ uniformly
  on compact subsets. Note that the limit nonlinearities $f^\pm$ are independent
  of the gradient $\grad u$. The limit problems
  \begin{align}
    \label{eq:dirichlet2}
    \partial_t u - \Delta u  &= f^\pm(x,u(t,x)) \\
    \notag  u(t,x) & = 0 & x\in\partial\Omega\\
    \notag  u(t,x) & = u_0(x) & x\in\Omega
  \end{align}
  define local gradient-like semiflows on an appropriate Banach space $\tilde X$.
  It is well-known that for generic $f^\pm$, every equilibrium
  of \eqref{eq:dirichlet2} is hyperbolic. Hence, a solution $u:\;\IR\to \tilde X$
  is either an equilibrium solution or a heteroclinic connection.
  
  It has been proved \cite{brunpol} 
  that for a generic $f$ the
  semiflow induced by
  \begin{align*}
    \partial_t u - \Delta u  &= f(x,u(t,x),0) \\
    \notag  u(t,x) & = 0 & x\in\partial\Omega\\
    \notag  u(t,x) & = u_0(x) & x\in\Omega
  \end{align*}
  is Morse-Smale. For the above equation, the Morse-Smale property means the following.
  \begin{enumerate}
    \item[(1)] Every bounded subset of $\tilde X$ contains only finitely many equilibria.
    \item[(2)] Given a pair $(e^-,e^+)$ of equilibria, the stable manifold $W^s(e^+)$ and
    the unstable manifold $W^u(e^-)$ intersect transversally.
  \end{enumerate}
  
  An easy consequence of property (2) is stated below.
  \begin{enumerate}
    \item[(2')] A connection\footnote{non-trivial, that is, except for constant solutions}
      between $e^-$ and $e^+$ can only exist if the respective
      Morse-indices satisfy $m(e^+)<m(e^-)$.
  \end{enumerate}

  The aim of this paper is to investigate if and how
  property (1) and (2') can be generalized
  to semilinear parabolic equations which are asymptotically autonomous,
  for example \eqref{eq:dirichlet}.
  Roughly speaking, the general situation is as follows: 
  Equilibria in the autonomous case
  correspond to connections between two equilibria having the same
  Morse-index, and every
  bounded set contains only finitely many such connections.
  Furthermore, a connection between equilibria $e^-$ and 
  $e^+$ can only exist if $m(e^+)\leq m(e^-)$.

  The proof of our results is similar to the relevant parts of \cite{brunpol},
  applying an abstract transversality theorem to a suitable
  differential operator. As a result, we know that for a dense subset of possible
  perturbations, $0$ is a regular value of this operator.
  
  Using the framework of \cite{brunpol}, namely the characterization
  of transversality in terms of the existence of exponential dichotomies
  on halflines \cite[Corollary 4.b.4]{brunpol}, we could try to prove that an appropriate generalization 
  of (2) (see \cite{carvalho2015non, czaja2017definition}) holds with respect to a perturbation for which the abstract
  differential operator has $0$ as a regular value. Following the approach of \cite{brunpol}, 
  we would have to assume that the evolution operator
  defined by the linearized equation at a heteroclinic solution is injective \cite[Lemma 4.a.12]{brunpol}.
  (1) and (2') can be proved to hold for a generic perturbation without
  the injectivity assumption. For this reason, (2) is replaced by (2').

  We will now apply Theorem \ref{th:140410-1618} to the
  concrete problem \eqref{eq:dirichlet}. Let $p>m\geq 1$, $X:=L^p(\Omega)$, which is reflexive,
  and define an operator 
  \begin{align*}
    A &:& &W^{2,p}(\Omega) \cap W^{1,p}_0(\Omega) \to L^p(\Omega) \\
    Au &:=& &-\Delta u.
  \end{align*}
  $A$ is a positive sectorial operator and has compact resolvent. As usual,
  define the fractional power space $X^\alpha$ as the range of 
  $A^{-\alpha}$ equipped with the norm $\norm{x}_\alpha := \norm{A^\alpha x}_X$.
  For $\alpha<1$ sufficiently large, the space $X^\alpha$ 
  is continuously imbedded in $C^1(\bar \Omega)$
  (see for instance \cite[Lemma 37.8]{sellyou}).
  Hence, $f$ gives rise to a Nemitskii operator $\hat f:\; \IR\times X^\alpha \to X$,
  where
  \begin{equation*}
    \hat f(t,u)(x) := f(t,x,u(x),\grad u(x)).
  \end{equation*}
  Suppose that for some $\delta>0$
  \begin{enumerate}
    \item $f(t,.)\to f^\pm$ uniformly on sets of the form $\Omega\times B_\eta(0)\times B_\eta(0)\subset\Omega\times\IR\times\IR^m$, 
      where $\eta>0$ and $f^\pm:\; \Omega\times\IR\to \IR$
      is continuously differentiable in its second variable with $\partial_u f^\pm(x,u)$ being continuous,
    \item $f(t,x,.,.)$ is $C^\infty$, and
    \item each partial derivative of $f(t,x,.,.)$ is continuous in $x$ and
      Hölder-continuous in $t$ with Hölder-exponent $\delta$
      uniformly on sets of the form $\IR\times\Omega\times B_\eta(0)\times B_\eta(0)\subset \IR\times\Omega\times\IR\times\IR^m$,
      $\eta>0$.
  \end{enumerate}

  Let $C^{0,\delta}_0(\IR\times\bar \Omega)$ denote the set of all in $t$ Hölder-continuous (with exponent $\delta>0$)
  functions $g:\; \IR\times\Omega\to \IR$ with $g(t,x)\to 0$ as $t\to\pm\infty$
  uniformly on $\Omega$. $C^{0,\delta}_0(\IR\times\bar \Omega)$ is endowed with the norm
  \begin{equation*}
    \norm{g} := \sup_{(t,x)\in\IR\times\Omega} \abs{g(t,x)} + \sup_{(t,x)\neq (t',x)\in\IR\times\Omega} \frac{\abs{g(t,x) - g(t',x)}}{\abs{t-t'}^\delta}.
  \end{equation*}

  \begin{theorem}
  	In addition to the hypotheses above, assume that every equlibrium of the equations
  	\begin{equation*}
	  	u_t + Au = \hat f^{\pm}(u)
  	\end{equation*}
  	is hyperbolic.
  	
    Then there is a residual subset $Y\subset C^{0,\delta}_0(\IR\times\bar\Omega)$ such that
    for all $g\in Y$ and for
	every bounded solution of $u:\; \IR\to W^{2,p}(\Omega)$ of
    \begin{equation*}
      u_t + Au = \hat f(t,u) + \hat g(t),
    \end{equation*}
    it holds that:
    \begin{enumerate}
      \item There are equilibria $e^\pm$ of 
        \begin{equation*}
          u_t + Au = \hat f^\pm(u)
        \end{equation*}
        such that $u(t)\to e^\pm$ in $C(\bar\Omega)$ as $t\to\pm\infty$.
      \item $m(e^-)\geq m(e^+)$ and $m(e^-) = m(e^+)$ only if
        \begin{equation*}
          v_t + Av = D\hat f(t,u(t))v
        \end{equation*}
        does not have a non-trivial ($L^p(\Omega)$-) bounded solution.
    \end{enumerate}
  \end{theorem}

  Since there are continuous imbeddings $X^1\subset C(\bar\Omega, \IR)\subset X^0$,
  the above theorem follows immediately from Corollary \ref{co:140602-1439}.
\end{section} 

\begin{section}{Abstract formulation of the result}
  Let $X$ and $Y$ be normed spaces and $X_0\subset X$ be open. $\mathcal{L}(X,Y)$
  is the space of all continuous linear operators $X\to Y$ endowed with
  the usual operator norm. The open ball with radius $\eps$ and center $x$ in $X$ is denoted by
  $B_\eps(x)$ and the closed ball with the same radius and center by $B_\eps[x]$.

  $C^k_B(X_0,Y)$ denotes the space of all $k$-times continuously differentiable
  mappings $X_0\to Y$ with bounded derivatives up to order $k$. The spaces
  are endowed with the usual norm
  \begin{equation*}
    \norm{y} := \sup_{x\in X_0} \max \{\norm{y(x)},\dots,\norm{D^ky(x)}\}
  \end{equation*}
  The space $C^{k,\delta}_B(X_0, Y)$ is the subspace of $C^k_B(X_0, Y)$
  consisting of all functions in $C^k_B(X_0, Y)$ whose $k$-order
  derivative is Hölder-continuous with exponent $\delta>0$. In the case $\delta=0$,
  we simply set $C^{k,0}_B(X_0, Y) := C^{k}_B(X_0, Y)$. 
  On $C^{k,\delta}_B(X_0, Y)$,
  we consider the norm 
  \begin{equation*}
    \norm{y} := \norm{y}_{C^k_B(\IR, X)} + \sup_{x,x'\in X_0\:x\neq x'} \frac{\norm{D^ky(x) - D^ky(x')}}{\norm{x-x'}^\delta}.
  \end{equation*}

  Let $\eta>0$ and $i_\eta:\; B_\eta(0)\cap X_0\to X_0$ the inclusion mapping. Let
  $C^{k,\delta}_b(X_0, Y)$ denote the
  set of all functions $f:\; X_0\to Y$ such that $f\circ i_\eta \in C^{k,\delta}_b(X_0\cap B_\eta(0), Y)$
  for all $\eta>0$.   We also write $C^k_b(X_0,Y) := C^{k,0}_b(X_0,Y)$ and $C_b(X_0,Y) := C^0_b(X_0,Y)$
  for short. These spaces are equipped with an invariant metric
  \begin{equation*}
    d(f,f') := d(f-f',0) := \sum_{n\in\IN} 2^{-n}\frac{\norm{(f-f')\circ i_\eta}}{1+\norm{(f-f')\circ i_\eta}}.
  \end{equation*}
  This metric induces the respective topology of uniform convergence on
  bounded sets, that is, $f_n\to f$ in $C^k_b(X_0, Y)$ (resp. $C^{k,\delta}_b(X_0, Y)$)
  if $f_n\circ i_\eps \to f\circ i_\eps$ in $C^{k,\delta}_B(X_0\cap B_\eps(0),Y)$ for every $\eps>0$.
  $C^{k,\delta}_{B,0}(\IR,Y)$ denotes the closed subspace of $C^{k,\delta}_B(\IR,Y)$ containing all
  functions $x$ with $x(t)\to 0$ as $\abs{t}\to\infty$.
 
  \begin{df}
    A family $T(t,s)$ defined for real numbers $t\geq s$
    of continuous linear operators is called a
    {\em linear evolution operator} if $T(r,t)T(t,s) = T(r,s)$
    for all $r\geq t\geq s$.
  \end{df}

  \begin{df}
    \label{df:140519-1839}
    We say that an evolution operator $T(t,s)$ on a normed space $X$
    admits an {\em exponential dichotomy}
    on an interval $J$ if there are constants $\gamma,M>0$
    and a family $(P(t))_{t\in J}$ in $\mathcal{L}(X,X)$
    such that:
    \begin{enumerate}
      \item $T(t,s)P(s) = P(t)T(t,s)$ for $t\geq s$.
      \item The restriction $T(t,s):\; \Ran(P(s))\to\Ran(P(t))$
        is an isomorphism. Its inverse is denoted by $T(s,t)$, where $s<t$.
      \item $\norm{T(t,s)(I-P(s))}_{\mathcal{L}(X,X)} \leq M e^{-\gamma (t-s)}$ for $t\geq s$.
      \item $\norm{T(t,s)P(s)}_{\mathcal{L}(X,X)}\leq M e^{\gamma (t-s)}$ for $t<s$.
    \end{enumerate}

    We also refer to the the family of projections as an exponential dichotomy.
  \end{df}

  \begin{df}
    \label{df:140509-1746}
    Let $\pi$ be a semiflow on a normed space $X$. We say that
    $\pi$ is {\em simple gradient-like} if:
    \begin{enumerate}
      \item[(a)] Every equilibrium $e$ of $\pi$ is isolated\footnote{The term {\em simple} refers to this hypothesis.}.
      \item[(b)] For every bounded solution $u:\;\IR\to X$, one has $u(t)\to e^-$
        as $t\to-\infty$ and $u(t)\to e^+$ as $t\to\infty$.
      \item[(c)] There is a partial order $\prec$ on the set $E$ of all equilibria
        such that $e^+\prec e^-$ whenever $u$ satisfies (b).
      \item[(d)] If $u$ is given by (b) and $e^- = e^+$, then $u\equiv e$.
    \end{enumerate}
  \end{df}
  
  Unless otherwise stated, let $X$ be a reflexive Banach space and
  $A$ a positive sectorial operator defined on subspace $X^1\subset X$.
  $X^\alpha := \Ran(A^{-\alpha})$ denotes the $\alpha$-th
  fractional power space with the norm $\norm{x}_\alpha := \norm{A^\alpha x}$.
  We will assume that the operator $A$ has compact resolvent.

  Fix some $\delta\in\left]0,1\right[$, and let $f\in C^{1,\delta}_b(\IR\times X^\alpha, X)$
  be asymptotically autonomous, that is, there are $f^\pm\in C^{1,\delta}_b(X^\alpha, X)$
  such that $f(t,.)\to f^\pm$ in $C^{1,\delta}_b(X^\alpha,X)$ as $t\to\pm\infty$. 
  We consider solutions of
  \begin{equation}
    \label{eq:140410-1557}
    u_t + Au = f(t,u)
  \end{equation}
  and its limit equations
  \begin{equation}
    \label{eq:140410-1558}
    u_t + Au = f^\pm(u).
  \end{equation}
  The above equations define evolution operators (respectively semiflows
  in the autonomous case) on $X^\alpha$.

  By an equilibrium $e$ of \eqref{eq:140410-1558}, we mean
  a point $e\in X^\alpha$ such that $u:\;\IR\to X^\alpha$, $t\mapsto e$,
  solves \eqref{eq:140410-1558}. We say that an equilibrium 
  $e$ is {\em hyperbolic} if the linearized equation
  \begin{equation*}
    u_t + Au = Df^\pm(e)u
  \end{equation*}
  admits an exponential dichotomy $(P(t))_{t\in\IR}$. The Morse-index of
  $e$ is the dimension of the exponential dichotomy, respectively the
  dimension of the range of its associated projection i.e., 
  $m(e) := \dim\Ran P(t)$, where $t\in\IR$ can be chosen
  arbitrarily. 

  \begin{theorem}
    \label{th:140410-1618}
    Assume that:
    \begin{enumerate}
      \item[(a)] Every equilibrium $e$ of \eqref{eq:140410-1558} is hyperbolic.
      \item[(b)] $f\in C_b(\IR\times X^\alpha, X)$
      \item[(c)] $f(t,.)\to f^\pm$ in $C^1_b(X^\alpha,X)$ as $t\to\pm\infty$.
      \item[(d)] $f(t,.)$ is $C^\infty$ for each $t\in\IR$, $t\mapsto D^kf(t,.)$
	      Hölder-continuous  with Hölder-exponent $\delta$ 
        uniformly on sets of the form $\IR\times B_\eta(0)\subset \IR\times X^\alpha$, 
        and $\norm{D^kf(t,x)} \leq C(k,\norm{x}_\alpha)$
        for all $k\in\IN\cup\{0\}$ and all $(t,x)\in\IR\times X^\alpha$.
      \item[(e)] The semiflows induced by \eqref{eq:140410-1558} are gradient-like.
    \end{enumerate}

    Let $\beta\in\left[0,1\right]$, and let $C^{0,\delta}_{B,0}(\IR, X^\beta)$ denote the complete subspace
    of all $x\in C^{0,\delta}_B(\IR, X^\beta)$ with $\norm{x(t)}_\alpha\to 0$
    as $\abs{t}\to\infty$.
    Then, for a generic\footnote{i.e., there is a residual subset of $C^{0,\delta}_B(\IR, X^\beta)$ such that all $g$
      in this subset have the stated property}
    $g\in C^{0,\delta}_{B,0}(\IR, X^\beta)$, every bounded
    solution $u:\;\IR\to X^\alpha$ of
    \begin{equation}
      \label{eq:140530-1823}
      u_t + Au = f(t,u) + g(t)
    \end{equation}
    satisfies:
    \begin{enumerate}
      \item[(1)] There are equilibria $e^-$, $e^+$ of the respective
        limit equation \eqref{eq:140410-1558} such that
        $\norm{u(t) - e^-}_\alpha \to 0$ as $t\to-\infty$ and
        $\norm{u(t) - e^+}_\alpha \to 0$ as $t\to\infty$.
      \item[(2)] $m(e^+)\leq m(e^-)$ and
        $m(e^-) = m(e^+)$ only if the linear equation
        \begin{equation}
          \label{eq:140410-1736}
          v_t + Av = Df(t,u(t))v
        \end{equation}
        does not have a non-trivial bounded solution $v:\;\IR\to X^\alpha$.
    \end{enumerate}
  \end{theorem}

  Note that (2) is equivalent to the existence of an exponential dichotomy
  for \eqref{eq:140410-1736} (cf. the proof of Lemma \ref{le:131219-1505}).

  \begin{proof}
    \begin{enumerate}
      \item Since the limit equations \eqref{eq:140410-1558}
        are gradient-like, this is a consequence of Lemma \ref{le:140307-1751}.
      \item This follows from Theorem \ref{th:131127-1503} together with 
        Lemma \ref{le:140107-1701}.
    \end{enumerate}
  \end{proof}

  \begin{corollary}
    \label{co:140602-1439}
    Let $E$ be a normed space such that $X^1\subset E\subset X^0$,
    the inclusions being continuous.

    Moreover, assume the hypotheses of Theorem \ref{th:140410-1618}.
    Then the conclusions of Theorem \ref{th:140410-1618} hold
    for a generic $g\in C^{0,\delta}_{B,0}(\IR, E)$.
  \end{corollary}

  \begin{proof}
  	Let $\Phi$ be defined as in Section \ref{sec:surjectivity}, preceding
	Theorem \ref{th:131127-1503}.
    Let $Y$ denote the set of all $g\in C^{0,\delta}_{B,0}(\IR, X)$ such
    that $0$ is a regular value of $\Phi(.,g)$.
    It follows from Theorem \ref{th:131127-1503} that $Y = \bigcap_{n\in\IN} Y_n$, where each $Y_n$
    is open and dense in $C^{0,\delta}_{B,0}(\IR, X)$.
    A second application of Theorem \ref{th:131127-1503} proves
    that $Y\cap C^{0,\delta}_{B,0}(\IR, X^1)$ is dense in $C^{0,\delta}_{B,0}(\IR, X)$.

    By the continuity of the inclusions, each of the sets
    $Y_n \cap C^{0,\delta}_{B,0}(\IR, E)$ is open in $C^{0,\delta}_{B,0}(\IR, E)$. Moreover,
    $Y\cap C^{0,\delta}_{B,0}(\IR, X^1)$ is a dense subset of each $Y_n \cap C_{B,0}(\IR, E)$,
    which proves that $\bigcap_{n\in\IN} (Y_n \cap C_{B,0}(\IR, E)) = Y\cap C_{B,0}(\IR, E)$
    is residual.
  \end{proof}
\end{section}

\begin{section}{A skew-product semiflow and convergence of solutions}
  Let $Y\subset C_b(\IR\times X^\alpha, X)$ denote the subspace, that is,
  equipped with a metric of convergence uniformly on bounded sets, of all 
  functions $f:\;\IR\times X^\alpha\to X$ such that:
  \begin{enumerate}
    \item $f(t,.)\in C^1_b(X^\alpha, X)$ for all $t\in\IR$
    \item $t\mapsto Df(t,.)$ is a Hölder-continuous function $\IR\to C^1_b(X^\alpha, X)$
  \end{enumerate}

  The above assumptions are rather strong, but we do not 
  strive for maximum generality here.
  It is easy to prove

  \begin{lemma}
    For every $f\in Y$, the translation $t\mapsto f^t(s,x) := f(t+s,x)$,
    $\IR\to Y$ is continuous.
  \end{lemma}

  Let $Y_0\subset Y$ be a compact subspace of $Y$ which is invariant
  with respect to translations. We consider solutions of the semilinear parabolic
  equation
  \begin{equation}
    \label{eq:semilin}
    \dot u + Au = y(t,u).
  \end{equation}
  These induce a skew-product semiflow $\pi:=\pi_{Y_0}$ on $Y_0\times X^\alpha$,
  where we set $(y,x)\pi t := (y^t, u(t))$ if there exits a solution $u:\;\left[0,t\right]\to X^\alpha$
  of \eqref{eq:semilin} with $u(0) = x$.
  It follows
  from \cite[Theorem 47.5]{sellyou} that $\pi$ is continuous.

  Now suppose that $y^t\to y^-$ as $t\to-\infty$ and $y^t\to y^+$ as
  $t\to\infty$, where $y^-, y^+\in Y$ are autonomous. It is easily
  seen that the set $Y_0 := \cl_Y \{y^t:\;t\in\IR\} = \{y^t:\;t\in\IR\}\cup\{y^-,y^+\}$
  is compact. Moreover for $y^+$ (resp. $y^-$), 
  \eqref{eq:semilin} defines a semiflow on $X^\alpha$,
  which is denoted by $\chi_{y^+}$ (resp. $\chi_{y^-}$).

  It is easy to see that the two lemmas still hold true in a more general setting, replacing
  the boundedness in $X^\alpha$ by an asymptotic convergence assumption, 
  admissibility \cite{ryb} for example.
  
  \begin{lemma}
    \label{le:140307-1751}
    Assume that $\chi_{y^+}$ (resp. $\chi_{y^-}$) is simple gradient-like,
    and let $u:\;\IR\to X^\alpha$ be a bounded solution of \eqref{eq:semilin}. Then,
    $u(t)$ converges to an equilibrium of $\chi_{y^+}$ (resp. $\chi_{y^-}$) 
    as $t\to\infty$ (resp. $t\to -\infty$).
  \end{lemma}

  In the following proof, we use as before infix notation for the semiflows, i.e. given
  an arbitrary semiflow $\pi$ on a metric space $X$, we write $x\pi t$
  instead of $\pi(t,x)$. A solution of $\pi$ or with respect to $\pi$
  is a continuous mapping $u:I\to X$ such that $I\subset\IR$ is an interval and
  $u(t) = u(t_0)\pi (t-t_0)$ whenever
  $\left[t_0,t\right]\subset I$. Given $N\subset X$, $\Inv^-_\pi(N)$ denotes the
  negatively invariant subset of $N$, i.e. $x\in \Inv^-_{\pi}(N)$ iff there exists
  a solution $u:\;\left]-\infty,0\right]\to N$ with $u(0) = x$.

  \begin{proof}
    We consider only the case $t\to\infty$ because $t\to-\infty$ can
    be treated analogously. Suppose to the contrary that $N\subset X^\alpha$
    is bounded and $u:\;\IR\to Y_0\times N$
    is a solution with $\omega(u)\neq\{(y^+,e_0)\}$, where $e_0$ denotes
    a minimal equilibrium in $\{x:\; (y^+,x)\in \omega(u)\}$. The minimality refers
    to the partial order $\prec$ introduced in Definition \ref{df:140509-1746}.

	Let $E\subset\{y^+\}\times X^\alpha$ denote the set of all equilibria in $\omega(u)$.
	Pick an $\eps>0$ such that $B_\eps[(y^+,e_0)]\cap E = \{(y^+,e_0)\}$ and
	a sequence $t_n\to\infty$ with $u(t_n)\to (y^+,e_0)$. 
	There are $s_n\geq t_n$
    such that $d(u(s_n), (y^+,e_0))=\eps$ and $u(\left[t_n,s_n\right])\subset B_\eps\left[(y^+,e_0)\right]$.
    
    We claim that $\abs{t_n-s_n}\to\infty$ as $n\to\infty$. Otherwise,
    we may assume without loss of generality that $r_n:=\abs{t_n-s_n}\to r_0$.
    The continuity of the semiflow implies that $\partial B_\eps[(y^+,e_0)]\ni u(s_n) \to (y^+,e_0)\pi r_0$,
    which is a contradiction.
    Choosing a subsequence $(s'_n)_n$ of $(s_n)_n$, we can assume that 
    $u(s'_n)\to (y^+,x_0)\in \partial B_\eps[(y^+,e_0)]\cap \Inv^-_{\pi}(B_\eps[(y^+,e_0)])$.
    Since $\chi_{y^+}$ is simple gradient-like, one has $(y^+,x_0)\pi t\to (y^+,e)$ as $t\to\infty$
    for some $e\in E$, in contradiction to the minimality of $e_0$.
  \end{proof}
\end{section}

\begin{section}{Surjectivity}
  \label{sec:surjectivity}
   
  The main result of this section is Theorem \ref{th:131127-1503}, 
  applying an abstract transversality theorem. One of the key steps
  towards its proof Theorem \ref{th:131206-1658} stating the surjectivity
  of certain linear operators. The main ingredient for the proof of Theorem \ref{th:131206-1658} is
  Lemma \ref{le:131206-1558}, which relies on a geometric idea that 
  can be sketched as follows. Let $u:\;\IR\to X^\alpha$ be a heteroclinic
  solutions that is, a solution connecting hyperbolic equilibria $e^-$ and $e^+$.
  The hyperbolicity of the equilibria implies the existence of exponential
  dichotomies on intervals of the form $\left]-\infty,\tau\right]$ and
  $\left[\tau,\infty\right[$ provided $\tau$ is large enough. The linear
  equation respectively its solution operators determines a connection
  between these dichotomies respectively their associated invariant spaces.
  Perturbing this connection is the idea behind Lemma \ref{le:131206-1558}.
  
  We consider the following (Banach) spaces:
  \begin{align*}
    \X &:= C^{1,\delta}_B(\IR,X) \cap C^{0,\delta}_B(\IR,X^1) \\
    \Y &:= C^{0,\delta}_{B,0}(\IR, X^\beta) := \{y\in C^{0,\delta}_B(\IR,X^\beta):\; y(t)\to 0\text{ as }t\to\pm\infty\}\quad 0\leq\beta\leq 1\\
    \Z &:= C^{0,\delta}_B(\IR,X).
  \end{align*}
  Here, we choose $\norm{x}_\X := \norm{x}_{C^{1,\delta}_B(\IR, X)} + \norm{x}_{C^{0,\delta}_B(\IR, X^1)}$.

  A function $f:\;\IR\times X^\alpha\to X^0$
  gives rise to a Nemitskii operator $\hat f$ defined by
  \begin{equation*}
    \hat f(u)(t) := f(t,u(t)).
  \end{equation*}

  \begin{lemma}
    Under the hypotheses (b), (c) and (d) of Theorem \ref{th:140410-1618},
    $\hat f$ maps bounded Hölder-continuous functions to bounded Hölder-continuous functions,
    that is, $\hat f(C^{0,\delta}_B(\IR,X^\alpha)) \subset C^{0,\delta}_B(\IR,X^0)$
  \end{lemma}

  \begin{lemma}
    \label{le:140514-1755}
    Under the hypotheses (b) and (d) of Theorem \ref{th:140410-1618}, the mapping
    $\hat f:\; C^{0,\delta}_B(\IR, X^\alpha)\to C^{0,\delta}_B(\IR, X)$ as defined above is $C^\infty$.
  \end{lemma}

  \begin{proof}
	Suppose that $u,u',v\in C^{0,\delta}_B(\IR, X^\alpha)$ satisfy $\norm{u},\norm{u'}\leq M$,
	and set $B(t) := \mathrm{D}_xf(t,u(t))$.
	
	By the assumptions on $\mathrm{D}_x f$ and $\mathrm{D}^2_x f$, 
	there are constants $C_1 := C_1(M)$ and $C_2 := C_2(M)$
	such that for arbitrary $v\in C^{0,\delta}(\IR,X^\alpha)$ and $t,s\in\IR^+$
	\begin{equation*}
	\begin{array}{lcl}
	\norm{B(t)v(t)}_{0} &\leq& C_1 \norm{v}_{C_B(\IR,X^\alpha)} \\
	\norm{B(t+s)v(t+s) - B(t)v(t)}_0 &\leq& C_2s^\delta\norm{v}_{C_B(\IR, X^\alpha)}
	+ C_1s^\delta\norm{v}_{C^{0,\delta}_B(\IR,X^\alpha)}.
	\end{array}
	\end{equation*}
	
	Now, set $B'(t,y):=\mathrm{D}_xf(t,u(t)+y) - \mathrm{D}_xf(t,u(t))$. We have
	\begin{equation*}
	\begin{array}{lcl}
	B'(t,y) &=& \int\limits^1_0 \mathrm{D}^2_x f(t,u(t)+\lambda y)y\de\lambda\\
	
	B'(t,y_1) - B'(t,y_2) &=&
	\left(\mathrm{D}_x f(t,x+y_1) - \mathrm{D}_x f(t,x)\right)\\
	&&- \left(\mathrm{D}_x f(t,x+y_2) - \mathrm{D}_x f(t,x)\right)\\
	&=& \int\limits^1_0 \mathrm{D}^2_x f(t,x+y_2+\lambda (y_1-y_2))(y_1-y_2)\de \lambda\\
	
	B'(t+s,y) - B'(t,y) &=& 
	\int\limits^1_0 \mathrm{D}^2_x f(t+s,u(t+s)+\lambda y)y\\
	&&-\mathrm{D}^2_x f(t,u(t)+\lambda y)y\de\lambda\\
	\end{array}
	\end{equation*}
	
	By the assumptions on $D^k_x f$, there are constants $C_3 := C_3(M)$
	and $C_4 := C_4(M)$ such that for all $y,y_1,y_2 \in B_M(0)\subset X^\alpha$
	and all $z, z_1, z_2\in X^\alpha$
	\begin{equation*}
	\begin{array}{lcl}
	\norm{B'(t,y)z}_0 &\leq& C_3 \norm{y}_\alpha\norm{z}_\alpha\\
	\norm{B'(t+s, y_2)z_2 - B'(t,y_1)z_1}_0 &\leq& C_4\bigl( \norm{y_2}_\alpha\norm{z_2-z_1}_\alpha 
	+ \norm{y_2-y_1}_\alpha \norm{z_1}_\alpha \\
	&&+ s^\delta \norm{y_1}_\alpha \norm{z_1}_\alpha + s^\delta \norm{u}_{C^{0,\delta}_B(\IR, X^\alpha)} \norm{y_1}_\alpha \bigl)\norm{z_1}_\alpha\\
	\end{array}
	\end{equation*}
	
	It follows that $[D\hat f(u)]v(t) := Df(t,u(t))v(t)$ satisfies
	\begin{equation}
	\label{eq:140515-1238}
	\norm{\mathrm{D} \hat f(u+u') - \mathrm{D} \hat f(u)}_{\mathcal{L}(C^{0,\delta}_B(\IR, X^\alpha), C^{0,\delta}_B(\IR, X))} \leq C_5(M) \norm{u'}_{C^{0,\delta}_B(\IR, X^\alpha)}.
	\end{equation}
	
	In particular, one has
	\begin{equation*}
	D\hat f\in C_b\big(C^{0,\delta}_B(\IR, X^\alpha),\;\; \mathcal{L}(C^{0,\delta}_B(\IR, X^\alpha), C^{0,\delta}_B(\IR, X))\big).
	\end{equation*}
	
	Moreover,
	\begin{equation*}
	f(t,x+y) = f(t,x) + \mathrm{D} f(t,x)y + \int\limits^1_0 \left(\mathrm{D} f(t,x+\lambda y) - \mathrm{D} f(t,x)\right)y\,\de \lambda ,
	\end{equation*}
	so
	\begin{equation*}
	\hat f(u+u') - \hat f(u) - \mathrm{D}\hat f(u)u' = \int\limits^1_0 \left(\mathrm{D}\hat f(u+\lambda u') - \mathrm{D}\hat f(u)\right) u'\,\de \lambda ,
	\end{equation*}
	and by \eqref{eq:140515-1238},
	\begin{equation*}
	\norm{\hat f(u+u') - \hat f(u) - \mathrm{D}\hat f(u)u'} \leq C_5(M) \norm{u'}^2,
	\end{equation*}
	which proves that $\hat f$ is continuously differentiable and $D\hat f$ as defined above
	is indeed the derivative. The higher derivatives can be treated analogously.
  \end{proof}

  Define $\Phi := \Phi_f:\; \X\times \Y \to \Z$ by 
  \begin{equation*}
    \Phi(u,g)(s) := u_t(s) + Au(s) - f(s,u(s)) - g(s).
  \end{equation*}
  $\Phi$ is continuous by the choice of $\X$, $\Y$, and $\Z$.

  Recall that a subset of a topological space is nowhere dense
  if the interior of its closure is empty. A countable union
  of nowhere dense sets is called meager and the complement of a meager set residual.
  The following theorem is the main result of this
  section.
  \begin{theorem}
    \label{th:131127-1503}
    Under the hypotheses of Theorem \ref{th:140410-1618}, 
    the set of all $y\in \Y$ such that
    $0$ is a regular\footnote{$D_x\Phi(x_0,y):\; \X\to \Z$ is surjective whenever $\Phi(x_0,y)=0$} value of $\Phi(.,y)$
    is residual (in $\Y$).
  \end{theorem}
  
  In order to prove Theorem \ref{th:131127-1503},
  we need to check the premises of the following theorem,
  which is a simplified version of \cite[Theorem 2.1]{brunpol}
  (see also \cite[Theorem 5.4]{henry2}).
  \begin{theorem}
    \label{th:140515-1501}
    Let $X,Y,Z$ be open subsets of Banach spaces, $r$ a positive integer,
    and $\Phi:\; X\times Y\to Z$ a $C^r$ map. 
    Assume that the following hypotheses are satisfied:
    \begin{enumerate}
      \item For each $(x,y)\in \Phi^{-1}(\{0\})$, $D_x\Phi(x,y):\; X\to Z$
        is a Fredholm operator of index less than $r$.
      \item For each $(x,y)\in \Phi^{-1}(\{0\})$ $D\Phi(x,y):\; X\times Y\to Z$
        is surjective.
      \item The projection $p:(x,y)\mapsto y:\; \Phi^{-1}(\{0\})\to Y$
        is $\sigma$-proper, that is, there is a countable system
        of subsets $V_n\subset \Phi^{-1}(\{0\})$ such that 
        $\bigcup_{n\in\IN} V_n = \Phi^{-1}(\{0\})$ and for
        each $n\in\IN$ the restriction $p_n:\; V_n\cap \Phi^{-1}(\{0\})\to Y$ of $p$
        is proper.
    \end{enumerate}
    
    Then the set of all $y\in Y$ such that $0$ is a regular
    value of $\Phi(.,y)$ is residual in $Y$.
  \end{theorem}
  
  Using Lemma \ref{le:140514-1755}, it is easy to see that $\Phi$ is $C^\infty$. In particular,
  we have
  \begin{equation*}
    D\Phi(u_0,v_0)(u,v) = u_t + Au - D\hat f(u_0)u - v.
  \end{equation*}

  Now, suppose that $\Phi(u_0,v_0) = 0$, that is, $u_0$ is
  a solution of 
  \begin{equation*}
    u_t + Au = \hat f(u) + v_0.
  \end{equation*}
  Under the assumptions of Theorem \ref{th:131127-1503},
  it follows from Lemma \ref{le:140307-1751}
  that $u(t)$ converges to a (hyperbolic) equilibrium $e^\pm$ of the
  respective limit equation as $t\to\pm\infty$. 

  \begin{proof}[Proof of Theorem \ref{th:131127-1503}]
    Initially, define
    \begin{equation*}
      \X_n := \{x\in \X:\; \norm{x(t)}_\alpha < n\;\text{ for all } t\in\IR\}\quad n\in\IN.
    \end{equation*}
    It is clear that $\X = \bigcup_{n\in\IN} \X_n$.

    Since each equilibrium of \eqref{eq:140410-1558} is hyperbolic, 
    there are only finitely many equilibria
    $e$ with $\norm{e}_\alpha \leq n$.
    Hence, there is an $m\in\IN$ such that $m(e)\leq m$
    whenever $e$ is an equilibrium of \eqref{eq:140410-1558}
    with $\norm{e}_\alpha \leq n$.

    Furthermore, there is an $\eps=\eps(n)>0$ such that
    $\norm{e-e'}_\alpha>2\eps$ for every pair $(e,e')$ of equilibria with
    $\norm{e}_\alpha\leq n$ and $\norm{e'}_\alpha\leq n$. Define
    \begin{equation*}
      \X_{n,m} := \{x\in \X_n:\; x(t)\in \bigcup_e B_\eps[e]\text{ for }\abs{t}\geq m\},
    \end{equation*}
    where the union is taken over all equilibria $e$ with $\norm{e}_\alpha\leq n$.
    
    Let $(u_0, y_0)\in \X_n\times \Y$ be a solution of $\Phi(u_0,y_0) = 0$.
    By our assumptions and \cite[Lemma 4.a.11]{brunpol}, assumption
    (CH) in Lemma \ref{le:131206-1558} is satisfied. Hence,
    it follows from Theorem \ref{th:131206-1658} and Lemma \ref{le:140327-1816}
    that for every solution $u_0\in \X_n$, $D_x\Phi(u_0,y_0):\; \X\to \Z$ is 
    a Fredholm operator and its (Fredholm) index is bounded by $m$. Furthermore,
    \begin{equation}
      \label{eq:140317-1627}
      L(u,v) := u_t + Au - D_u f(t,u_0)u + v
    \end{equation}
    defines a surjective operator $\X\times W\to \Z$,
    where $W = \spn\{w_1, \dots, w_m\}$ and $w_1,\dots, w_m\in \Y$
    have compact support. 

    In order to apply Theorem \ref{th:140515-1501}, we need
    to show that the map $(x,y)\mapsto y:\; \Phi^{-1}(\{0\})\to \Y$ is
    $\sigma$-proper, that is, there is a family $(V_n)_n$ with
    $\Phi^{-1}(\{0\}) = \bigcup_{n\in\IN} V_n$ such that for each $n\in\IN$
    the map 
    \begin{equation}
      \label{eq:131219-1128}
      (x,y)\mapsto y:\; V_n \to \Y
    \end{equation} 
    is proper.

    Let $(x,y)\in \Phi^{-1}(\{0\})$ with $x\in \X_n$. Since $y(t)\to 0$
    as $t\to\pm\infty$, $x$ converges to an equilibrium as $t\to\pm\infty$ (Lemma \ref{le:140307-1751}).
    Hence, 
    \begin{equation*}
       \Phi^{-1}(\{0\}) = \bigcup_{(n,m)\in\IN\times\IN} \big(\Phi^{-1}(\{0\})\cap (\underbrace{\X_{n,m}\times \Y}_{=:V_{n,m}})\big).
    \end{equation*}
    Let
    $(x_n, y_n)$ be a sequence in $V_{n,m}$ with $y_n\to y_0$ in $\Y$.
    Using the compactness of the evolution operators on $X^\alpha$ defined
    by 
    \begin{equation*}
      u_t + Au = f(t,u) + y\qquad y\in \Y,
    \end{equation*}
    it follows that there is a solution $x_0:\;\IR\to X^\alpha$ and a subsequence
    $(x'_n)_n$ such that $x'_n\to x_0$ uniformly on bounded sets. Suppose that
    the convergence is not uniform with respect to $t\in\IR$. 
    In this case, there are a subsequence $(x''_n)_n$,
    a sequence $(t_n)_n$ and an $\eta>0$ such that $\norm{x''_n(t_n)-x_0(t_n)}\geq \eta$
    for all $n\in\IN$. Moreover, we can assume without loss of generality
    that $t_n\to\infty$ or $t_n\to-\infty$.

    By the choice of $V_{n,m}$, there are equilibria $e^\pm$
    with $x''_n(t)\in B_\eps(e^\pm)$ for all $t$ with $\abs{t}\geq m$. Hence, one has
    $x_0(t)\in B_\eps[e^\pm]$ for $\abs{t}\geq m$. 
	Using
    assumption (c) of Theorem \ref{th:140410-1618} and \cite[Theorem 47.5]{sellyou},
    it follows that there is a solution $u:\IR\to B_\eps[e]$ (either $e=e^+$ or $e=e^-$) of
    one of the limit equations such that $\norm{u(0)-e}_\alpha\geq\eta>0$.
    We can assume without loss of generality that $B_\eps[e]$ is
    an isolating neighborhood for $e$, which means that $u\equiv e$.
    This is an obvious contradiction, so 
    \begin{equation*}
      \sup_{t\in\IR} \norm{x_n(t)-x_0(t)}_\alpha \to 0\text{ as }n\to\infty.
    \end{equation*}

    By \cite[Lemma 4.a.6]{brunpol},
    one has $x_n\to x_0\in\X$, which proves that the map defined by \eqref{eq:131219-1128}
    is proper.
    
    Now, it follows from Theorem \ref{th:140515-1501} that there
    is a residual subset $\Y_n \subset \Y$ such that for every $y\in \Y_n$,
    $0$ is a regular value of $\Phi(.,y):\; \X_n\to\Z$.
    
    This completes the proof since a countable intersection of residual sets is 
    residual.
  \end{proof}

  \begin{lemma}
    \label{le:131206-1352}
    For every $F\in\mathcal{L}(\IR^n,\IR^n)$ with
    $\det F>0$, there is an $\hat F\in C^{\infty}(\left[0,1\right], \mathcal{L}(\IR^n,\IR^n))$
    such that $\hat F(0) = \id$, $\hat F(1) = F$, and $\det F(t)>0$ for all $t\in\left[0,1\right]$.
  \end{lemma}
  
  The proof is omitted.

  \begin{lemma}
    \label{le:131206-1558}
    Suppose that:
    \begin{enumerate}
      \item[(CH)] 
        $B\in C^{0,\delta}(\mathcal{L}(X^\alpha,X))$ with $B(t)\to B^+$ 
        as $t\to\infty$ and $B(t)\to B^-$ as $t\to-\infty$.
        There further exists an $m^+\in\IN$ (resp. $m^-$) such that 
        the evolution operator defined by solutions of
        \begin{equation}
          \label{eq:131205-1729}
          u_t + Au = B^+ u \text{ (resp. $B^- u$) }
        \end{equation}
        admits an exponential dichotomy $P$ defined for $t\in\IR^+$ (resp. $t\in\IR^-$) with 
        $\abs{t}$ large and $\dim \Ran(P) = m^+$ (resp. $m^-$).
    \end{enumerate}

    If $m^- = m^+ =: m$, then there exist $t_1\leq t_2$
    and an $R\in C^{\infty}(\left[t_1,t_2\right],\mathcal{L}(X^\alpha, X))$ such that 
    there does not exist a bounded non-trivial (mild) solution of
    \begin{equation}
      \label{eq:131206-1213}
       u_t + Au = B(t)u + \begin{cases}
        R(t)u & t\in\left[t_1,t_2\right]\\
        0 & \text{otherwise.}
      \end{cases}
    \end{equation}
  \end{lemma}

  \begin{lemma}
    \label{le:131216-1436}
    Suppose that $A$ is a positive sectorial operator having
    compact resolvent. Let $X_1\subset X^1 = \Def(A)$ be an arbitrary
    finite-dimensional subspace.

    Then, there are a closed subspace $X_2\subset X$ and $B'\in\mathcal{L}(X,X)$
    such that $X = X_1\oplus X_2$, $(A-B')x = 0$ for 
    all $x\in X_1$, and $(A-B')x \in X_2$ for all $x\in X_2\cap \Def(A)$.
  \end{lemma}

  \begin{proof}
  	The claim is trivial for $X_1=\{0\}$, so we will assume
  	that $X_1\neq\{0\}$.
  	
  	Let $P\in\mathcal{L}(X,X_1)$ denote an otherwise arbitrary projection, and 
  	let $R(\mu,A)\in\mathcal{L}(X,X)$ denote the resolvent of $A + \mu I$.
    We have \cite[Theorem 5.2 in Chapter 2]{pazy}
    \begin{equation*}
      \norm{R(\mu,A)} \leq \frac{M}{\abs{\mu}},
    \end{equation*}
    so every real $\mu>0$ sufficiently large is
    in the resolvent set of
    \begin{equation}
      \label{eq:131216-1453}
      A + \mu I - AP = (A+\mu I)(I - R(\mu,A)AP).
    \end{equation}

    Moreover, the resolvent $R'(\mu)$ of \eqref{eq:131216-1453}
    is compact, and $\frac{1}{\mu}$ is an eigenvalue 
    of $R'(\mu)$. Let $X = X'_1 \oplus X'_2$ be the associated
    decomposition of $X$, where $X'_1\supset X_1$ is the generalized eigenspace
    associated with $\frac{1}{\mu}$ and $X'_2$ is $R'(\mu)$ invariant.

    Finally, let $Q\in\mathcal{L}(X,X'_1)$ denote the projection
    with kernel $X'_2$. The operator 
    \begin{equation*}
      A - \underbrace{\left( AP + A(I-P)Q \right)}_{=:B'}
    \end{equation*}
    vanishes on $X'_1$. Let $C$ satisfy the relation
    $X'_1 = X_1 \oplus C$, and set $X_2 := C\oplus X'_2$.
  \end{proof}

  \begin{proof}[Proof of Lemma \ref{le:131206-1558}]
    \renewcommand{\S}{\mathcal{S}}
    \newcommand{\U}{\mathcal{U}}

    Let the evolution operator $T(t,s)$ be defined by
    \begin{equation*}
      u_t + Au = B(t)u,
    \end{equation*}
    and consider the bundles
    \begin{align*}
      \U := \{&(s,x)\in\IR\times X^\alpha:\; \text{there exists a solution }u:\;\IR^-\to X\text{ with }u(s)=x\\
      &\text{ and }\sup_{t\in\IR^-} \norm{u(t)}_\alpha<\infty\}\\
      \S := \{&(s,x)\in\IR\times X^\alpha:\; \sup_{t\in\IR^+} \norm{T(t,s)x}_\alpha < \infty\}.
    \end{align*}

    $\U$ and $\S$ are positively invariant, that is, $(s,x)\in \U$ (resp. $\S$) implies
    $(t,T(t,s)x)\in \U$ (resp. $\S$) for all $t\geq s$.

    It is well-known that, for small $t\in\IR$ (resp. large $t\in\IR$), 
    $\dim \U(t) = m$ and $\codim \S(t) = m$ (see for instance \cite[Lemma 4.a.11]{brunpol}).
    Choose $t_1<t_2$ such that $\dim \U(t) = m$ for all $t\leq t_1$
    and $\codim \S(t) = m$ for all $t\geq t_2$. 

    Let $X = \S(t_2) \oplus C_{\S}$, $X_1 := \U(t_1) + C_\S$,  and $X = X_1\oplus X_2$
    with $X_2\subset \S(t_2)$. For $t\geq s\geq t_2$, the evolution
    operator $T(t,s)$ induces an isomorphism $X/\S(s) \to X/\S(t)$, so
    $X = T(t,t_2)C_\S \oplus \S(t+t_2)$ for every $t\in\IR^+$. 
    By standard regularity results and choosing
    $t_2$ larger if necessary, we can thus assume without loss of generality that
    $C_\S\subset X^1$ so that $X_1 = \U(t_1) + C_\S \subset X^1$.

    Let $F:\;X_1\to X_1$ be a linear endomorphism with $\det F>0$ which
    takes $\U(t_1)$ to $C_\S$, let $\hat F$ be given by Lemma \ref{le:131206-1352},
    and set $G(t_1 + \xi (t_2 - t_1)) := \hat F(\xi)$ for $\xi\in\left[0,1\right]$.
    Let $B'$ be defined by Lemma \ref{le:131216-1436}, and let $X = X_1 \oplus \tilde X_2$
    with an $(A-B')$-invariant complement $\tilde X_2$. $\tilde P\in\mathcal{L}(X,X_1)$ denotes
    the projection along $\tilde X_2$. Consider the semigroup $S(t)$
    defined by
    \begin{equation*}
      \dot u + Au = B'u.
    \end{equation*}

    We can now define the modified evolution operator $\hat T(t,s)$ by
    \begin{equation*}
      \hat T(t,s)(x) :=
      \begin{cases}
        G(t)G(s)^{-1} x & x\in X_1\text{ and }\left[s,t\right]\subset \left[t_1, t_2\right]\\
        S(t-s) x & x\in \tilde X_2\text{ and }\left[s,t\right]\subset \left[t_1, t_2\right]\\
        T(t,s) x & \left[s,t\right]\cap \left]t_1, t_2\right[ = \emptyset.
      \end{cases}
    \end{equation*}
    
    One has $\hat T(t_2,t_1)x = F(x)$ for all $x\in\U(t_1)$,
    so $\hat T(t_2, t_1)\U(t_1) \subset C_\S$, which proves that there does not exist
    a full bounded solution of $\hat T$.

    Assume that $u$
    is a solution of $\hat T$ defined for $t\in\left]a,b\right[\subset\left[t_1,t_2\right]$.
    We have
    \begin{align}
       \tilde Pu_t 
         \label{eq:171130-1421a}
			       &= \underbrace{(-A + B')\tilde P u}_{=0} + G_t(t)G(t)^{-1} \tilde Pu \\
      \label{eq:171130-1421b}
      (1-\tilde P)u_t &= (-A + B') (1-\tilde P)u, 
    \end{align}
    where the term $(-A + B')\tilde P u$ has been added deliberately.
    Consequently, every solution of $\hat T(t,s)$ is also a solution of \eqref{eq:131206-1213},
    where 
    \begin{equation*}
      R(t) := B' + G_t(t)G(t)^{-1}\tilde P - B(t)
    \end{equation*}
    is obtained by comparing the sum of \eqref{eq:171130-1421a} and \eqref{eq:171130-1421b}
    with \eqref{eq:131206-1213}.
  \end{proof}

  \begin{lemma}
    \label{le:131219-1505}
    Let $B\in L^\infty(\IR, \mathcal{L}(X^\alpha, X))$
    with $B(t)\to B^+$ 
    as $t\to\infty$ and $B(t)\to B^-$ as $t\to-\infty$.

    Assume there exists an $m\in\IN$ such that each of
    the evolution operators defined by solutions of
    \begin{equation*}
      u_t + Au = B^+ u \text{ (resp. $B^- u$) }
    \end{equation*}
    admits an exponential dichotomy $P$ defined for $t\in\IR^+$ (resp. $t\in\IR^-$) with 
    $\abs{t}$ large and $\dim \Ran(P) = m$.

    Moreover, suppose that the only bounded mild solution $u:\;\IR\to X^\alpha$
    of 
    \begin{equation}
      \label{eq:140325-1535}
      u_t + Au = B(t)u
    \end{equation}
    is $u\equiv 0$.
    
    Then, for every $h\in L^\infty(\IR, X)$, there is a unique 
    mild solution $u_0\in C_B(\IR,X^\alpha)$ of 
    \begin{equation}
      \label{eq:140325-1530}
      u_t + Au = B(t)u + h.
    \end{equation}
  \end{lemma}

  \begin{proof}
    It follows from \cite[Theorem 44.3]{sellyou} that \eqref{eq:140325-1530}
    generates a skew-product semiflow on a suitable phase space $W\times X^\alpha$,
    where $W := \cl \{B(t):\; t\in\IR\}$, $p$ is a sufficiently large integer,
    and the closure is taken in $L^p_\loc(\IR, \mathcal{L}(X^\alpha, X))$.
    Note that $W = \{\hat B^-, \hat B^+\} \cup \{B(t):\; t\in\IR\}$, where
    $\hat B^\pm(t)\equiv B^\pm$.

    Now \cite[Theorem C]{sacker_sell_dich} implies that the evolution operator $T(t,s)$
    defined by mild solutions of \eqref{eq:140325-1535} 
    admits an exponential dichotomy. Our claim follows using the
    same formula as \cite[Theorem 7.6.3]{henry} (see also \cite[Lemma 4.a.7]{brunpol} and \cite[Lemma 4.a.8]{brunpol}).
  \end{proof}
    
  \begin{theorem}
    \label{th:131206-1658}
    Suppose that (CH) holds, and let $m:=\max\{m^-,m^+\}$. Then there are 
    $w_1,\dots, w_m\in \Y$ having compact support 
    such that the operator $\tilde L:\; \X + \spn\{w_1,\dots, w_m\} \to \Z$
    \begin{equation*}
      \tilde L(u,w) := L(u,w) = u_t + Au - B(t)u - w
    \end{equation*}
    is surjective.
  \end{theorem}

  \begin{proof}
    Consider the spaces
    \begin{align*}
      X' &:= \IR^{\abs{m^--m^+}} \times X\\
      (X')^\alpha &:= \IR^{\abs{m^--m^+}} \times X^\alpha\\
    \end{align*}
    and
    \begin{align*}
      \X' &:= C^{1,\delta}_B(\IR,X') \cap C^{0,\delta}_B(\IR,(X')^1) \\
      \Y' &:= C^{0,\delta}_{B,0}(\IR, \IR^{\abs{m^--m^+}}) \times \Y\\
      \Z' &:= C^{0,\delta}_B(\IR,X').
    \end{align*}
    We define an operator $L':\;\X'\times \Y'\to\Z'$, where
    \begin{equation*}
      L'((x,u),(w',w)) := (x_t + \mu \arctan(t)x - w', u_t + Au -B(t) u -w)
    \end{equation*}
    and $\mu = 1$ if $m^-<m^+$ and $\mu=-1$ otherwise.
    
    In both cases and for both limit equations i.e., $t\to\pm\infty$,
    $(0,0)$ is an equilibrium 
    having Morse index $m$. It follows easily that $L$ is surjective if
    we prove that $L'$ is surjective.
    
    For the sake of simplicity, we will henceforth assume that $m^-=m^+$.

    By Lemma \ref{le:131206-1558}, there are $t_1\leq t_2$
    and $R\in C^\infty(\left[t_1,t_2\right], \mathcal{L}(X^\alpha, X))$
    such that 
    \begin{equation}
      \label{eq:131206-1717}
      u_t + Au - B(t)u =
      \begin{cases}
        R(t)u &t\in\left[t_1,t_2\right]\\
        0 &\text{otherwise}
      \end{cases}
    \end{equation}
    does not have a non-trivial bounded solution. 
    The evolution operator $T(t,s)$
    defined by \eqref{eq:131206-1717} has an exponentially stable subspace of finite
    codimension for $t\geq s\geq t_2$, that is, $X = X_1\oplus X_2$
    with $\codim X_2 = m^+$ and for some $M,\delta>0$
    \begin{equation}
      \label{eq:140321-1620}
      \norm{T(t,t_2)x}_\alpha \leq Me^{-\delta(t-s)}\norm{x}_\alpha \text{ for }x\in X_2\text{ and }t\geq t_2.
    \end{equation}
    Suppose that $\tilde X := T(t_2, t_1)X^1 \subset X^1$, $\tilde X_2 := \tilde X\cap X_2$
    and $\tilde X = \tilde X_1 \oplus \tilde X_2$.

    \begin{sublemma}
      \label{le:140321-1534}
      For every $\eta\in \tilde X_1$, there is a $w\in \Y$ and
      a solution $v:\;\left[t_1, t_2\right]\to X^\alpha$ of
      \begin{equation*}
        v_t + Av = B(t)v + w
      \end{equation*}
      with $v(t_1) = 0$ and $v(t_2)=\eta$.
    \end{sublemma}

    \begin{proof}
      Let $u:\;\left[t_1,t_2\right]\to X^\alpha$ be a solution of $T(t,s)$
      with $u(t_2) = \eta\neq 0$. Note that, by standard regularity results, e.g. \cite[Lemma 4.a.6]{brunpol},
      one has $u\in C^{1,\delta}(\IR, X) \cap C^{0,\delta}(\left[t_1,t_2\right], X^1)$.

      Let $x:\;\IR\to \IR$ be $C^\infty$ with $x(t) = 0$ for $t\leq t_1$ and $x(t)=1$ for $t\geq t_2$.
      Setting $v(t):=u(t)\cdot x(t)$, one has
      \begin{align*}
        v_t(s) &= u_t(s) \cdot x(t) + u(t) \cdot x_t(s)\\
        &= (-A + B(s)) \underbrace{u(s) x(s)}_{=v(s)} + \underbrace{u(s) x_t(s)}_{=:w(s)}
      \end{align*}
      and
      \begin{equation*}
        \begin{array}{lcl}
          v(t_1) &=& x(t_1) \cdot\eta = 0\\
          v(t_2) &=& x(t_2)\cdot \eta = \eta
        \end{array}
      \end{equation*}
      as claimed.
    \end{proof}

    Let $\eta_1,\dots,\eta_n$ be a basis for $\tilde X_1$, and choose
    $w_1,\dots, w_n$ and $v_1,\dots, v_n$ according to Sublemma \ref{le:140321-1534}.
    It follows from Lemma \ref{le:131219-1505} that for
    every $h\in\Z$, there exists a unique mild solution $u_0\in C_B(\IR,X^\alpha)$ of
    \begin{equation}
      \label{eq:131206-1747}
      u_t + Au - B(t)u = R(t)u + h.
    \end{equation}

    Let $v_1:\;\left[t_1,\infty\right[\to X^\alpha$ denote the solution
    of
    \begin{equation*}
      v_t + Av - B(t)v = R(t)u_0\quad v(t_1) = 0,
    \end{equation*}
    and let $v_1(t_2) = \eta \oplus \eta' \in \tilde X_1\oplus \tilde X_2$.
    
    There is a $w_0\in\spn\{w_1,\dots,w_n\}$ such that the solution
    $v_2:\;\left[t_1,\infty\right[\to X^\alpha$ of
    \begin{equation*}
       v_t + Av - B(t)v = -w_0\quad v(t_1) = 0
    \end{equation*}
    satisfies $v_2(t_2) = \eta$.

    It follows that $v_0 := v_1 - v_2$ is a solution of
    \begin{equation*}
      v_t + Av - B(t)v = R(t)u_0 + w_0 \quad v(t_1) = 0
    \end{equation*}
    with $v_0(t_2)\in\tilde X_2\subset X_2$.

    Using \eqref{eq:140321-1620}, one concludes that $\sup_{t\in\IR}\norm{v_0(t)}_\alpha <\infty$.
    Furthermore, $u_0-v_0$ is a bounded mild solution of
    \begin{equation*}
      u_t + Au - B(t)u - w_0 = h,
    \end{equation*}
    so by \cite[Lemma 4.a.6]{brunpol}, one has $u_0-v_0\in\X$ and thus
    $L(u_0-v_0,w_0) = h$, which completes the proof of Theorem \ref{th:131206-1658}.
  \end{proof}

  \begin{lemma}
    \label{le:140327-1816}
    Suppose that $A$ is a sectorial operator having compact
    resolvent and $B$ satisfies (CH). Let the operator $L := L_B$
    be defined by 
    \begin{equation*}
      L_B u := u_t + Au - B(t)u
    \end{equation*}
    Then $\dim \Null(L_B) \leq m^-$.
  \end{lemma}

  \begin{proof}
    This is an immediate consequence of the existence 
    of an exponentional dichotomy on an interval $\left]-\infty,t_0\right]$
    for small $t_0$, which follows from \cite[Lemma 4.a.11]{brunpol}.
  \end{proof}
    
\end{section}

\begin{section}{Adjoint equations}
  Throughout this section, suppose that $X$ is a reflexive Banach space,
  $A$ is a positive sectorial operator defined on $X^1\subset X$. 
  As usual, we write $\spr{x,x^*} := x^*(x)$. The
  adjoint operator $A^*$ with respect to this pairing
  is a positive sectorial operator on the
  dual space $X^*$ \cite[Theorem 1.10.6]{pazy}. Let $A^{*,\alpha}$
  denote the $\alpha$-th fractional power of the operator $A^*$
  and $X^{*,\alpha}$ the $\alpha$-th fractional power space
  defined by $A^{*,\alpha}$.

  For the rest of this section, fix some $\alpha\in\left[0,1\right[$,
  and suppose that (CH) holds. Recall that (CH) means in particular that
  $B(t)\to B^{\pm}$ as $t\to\pm\infty$. We also write $B(\pm\infty)$ to
  denote $B^\pm$.

  We will exploit the relationship between 
  \begin{equation}
    \label{eq:140106-1522}
    u_t + Au = B(t)u
  \end{equation}
  and its adjoint equation, where the adjoint
  is taken formally with respect to the pairing $(x,y) := \spr{x, A^{*,\alpha} y}$
  between $X$ and $X' := X^{*,\alpha}$.
  The adjoint equation for
  \eqref{eq:140106-1522} reads as follows.
  \begin{equation}
    \label{eq:140106-1522b}
    v_t + A^*v = (B(-t)A^{-\alpha})^* A^{*,\alpha} v =: B'(t)v
  \end{equation}

  \begin{lemma}
    \label{le:140407-1609}
    \begin{equation*}
       \spr{x,A^{*,\alpha}y} 
       = \spr{x,(A^*)^{\alpha}y}
       = \spr{x,(A^\alpha)^*y} \quad\forall (x,y)\in X^\alpha\times X^{*,\alpha}
    \end{equation*}
  \end{lemma}

  \begin{proof}
  	Recall that $A^{*,\alpha} = (A^*)^\alpha$ by definition.
    We have \cite[p. 70]{pazy} 
    \begin{equation*}
      A^{-\alpha} = \frac{1}{\Gamma(\alpha)} \int\limits^\infty_0 t^{\alpha-1} e^{-At}\,\de t,
    \end{equation*}
    where the integral is taken in $\mathcal{L}(X,X)$.

    Hence, for $x\in X$ and $y\in X^*$, one has
    \begin{align*}
      \spr{A^{-\alpha} x, y} 
      &= \frac{1}{\Gamma(\alpha)} \int\limits^\infty_0 t^{\alpha-1} \spr{e^{-At}x,y}\,\de t \\
      &= \frac{1}{\Gamma(\alpha)} \int\limits^\infty_0 t^{\alpha-1} \spr{x, e^{-A^*t}y}\,\de t\\
      &= \spr{x, (A^*)^{-\alpha}y}.
    \end{align*}
  \end{proof}

  \begin{lemma}
    \label{le:140108-1611}
    Let $B\in \mathcal{L}(X^\alpha,X)$. Then
    $B' := (BA^{-\alpha})^* A^{*,\alpha} \in \mathcal{L}(X^{*,\alpha},X^*)$
    with $\norm{B'} \leq \norm{B}$.
  \end{lemma}

  \begin{proof}
    Let $(x,y)\in X\times X^{*,\alpha}$. We have
    \begin{equation*}
      \begin{split}
        \abs{\spr{x,B'y}}
        &= \abs{\spr{BA^{-\alpha}x,A^{*,\alpha}y}} \\
        &\leq \norm{B}_{\mathcal{L}(X^\alpha,X)} \norm{x}_X \norm{A^{*,\alpha}y}_{X^*},
      \end{split}
    \end{equation*}
    which shows that $\norm{B'y}_{X^*} \leq \norm{B}_{\mathcal{L}(X^\alpha,X)} \norm{y}_{X^{*,\alpha}}$.
  \end{proof}

  \begin{lemma}
    \label{le:140108-1619}
    Let $J\subset\IR$ be an open interval, let $u:\;J\to X^\alpha$
    be a solution of \eqref{eq:140106-1522} and $v:\;-J\to X^{*,\alpha}$ be
    a solution of \eqref{eq:140106-1522b}. Then 
    \begin{equation*}
      (u(t),v(-t)) \equiv C\quad\text{for all } t\in J.
    \end{equation*}
  \end{lemma}

  \begin{proof}
    We consider the function $h(t) := (u(t),v(-t))$, which is
    defined for all $t\in J$. Note that $B$ is Hölder-continuous by (CH).
    Lemma \ref{le:140108-1611} implies that $B'$ is also Hölder-continuous.
    Therefore, $u$ and $v$ are continuously differentiable in $X$ respectively $X^*$.
    One has
    \begin{equation*}
      \begin{split}
        h_t(s) &= \lim_{h\to 0} \frac{1}{h}\left( \spr{u(s+h) - u(s),A^{*,\alpha}v(-s-h)} + \spr{A^\alpha u(s), v(-s-h) - v(-s)}\right)\\
        &= \spr{u_t(s), A^{*,\alpha}v(-s)} + \spr{A^\alpha u(s), -v_t(-s)} \\
        &= (-Au(s) + B(t)u(s), v(-s)) + \spr{A^\alpha u(s), A^*v(-s) - B'(-t)v(-s)} = 0\\
      \end{split}
    \end{equation*}
  \end{proof}

  \begin{lemma}
    \label{le:140409-1425}
    Let $J\subset \IR$ be an interval and $P:\; J\to \mathcal{L}(X^\alpha, X^\alpha)$ 
    an exponential dichotomy for the evolution operator $T(t,s)$ on $X^\alpha$ defined
    by \eqref{eq:140106-1522}. 

    Then $P':\; -J\to \mathcal{L}(X^{*,\alpha}, X^{*,\alpha})$, $P'(t) := A^{*,-\alpha}P(-t)^*A^{*,\alpha}$,
    is an exponential dichotomy for the evolution operator $T'(t,s)$
    defined by \eqref{eq:140106-1522b}.
  \end{lemma}

  \begin{proof}
    It is easy to see that $P'$ is well-defined and continuous (Lemma \ref{le:140108-1611}).
    We need to check the assumptions of an exponential dichotomy (Definition \ref{df:140519-1839}).

    Suppose that
    $(x,y)\in X^\alpha\times X^{*,\alpha}$
    and $\left[s,t\right]\subset J$.
    \begin{enumerate}
      \item From Lemma \ref{le:140108-1619}, we obtain
        \begin{equation*}
          \begin{split}
            (x,P'(-s)T'(-s,-t)y) &= (T(t,s)P(s)x, y)\\
            &= (P(t)T(t,s)x, y) \\
            &= (x, T'(-s,-t)P'(-t)y).
          \end{split}
        \end{equation*}
        Since $A^\alpha:\; X^\alpha\to X$ is an isomorphism, it follows
        that $P'(-s)T'(-s,-t) = T'(-s,-t)P'(-t)$.
      \item To show that $T'(-s,-t):\;\Ran(P'(-t))\to \Ran(P'(-s))$
        is an isomorphism, it is sufficient to show that it is injective.
        Suppose that $T'(-s,-t)y = 0$ for some $y\in\Ran(P'(-t))$. For $x\in X^\alpha$,
        we have
        \begin{equation*}
          0 = (x,T'(-s,-t)y) = (T(t,s)x,P'(-t)y) = (T(t,s)P(s)x,y),
        \end{equation*}
        so $(x,y) = 0$ for all $x\in \Ran(P(t))$. This in turn implies
        $(x,y) = (x,P'(-t)y) = (P(t)x,y) = 0$ for all $x\in X^\alpha$, that is,
        $y = 0$.
      \item
        The estimates for $y\in \Ran(P'(-t))$ and $y\in \Ran(I-P'(-t))$ can be deduced
        using roughly the same arguments. Hence, we will treat only the case
        $y\in \Ran(P'(-t))$.

        Suppose that 
        \begin{equation*}
          \norm{T(t,s)x}_\alpha \leq M e^{-\gamma(s-t)} \norm{x}_\alpha\quad s>t\quad x\in\Ran(P(s)).
        \end{equation*}
        We have
        \begin{equation*}
          \begin{split}
            \spr{x, A^{*,\alpha}T'(-s,-t)y}
            &= (x,T'(-s,-t)P'(-t)y) \\
            &= (P(s)x, T'(-s,-t)P'(-t)y) \\
            &= (T(t,s)P(s)x, y) \\
            & \leq CM e^{-\gamma(s-t)} \norm{x}_X \norm{A^{*,\alpha}y}_{X^*}.
          \end{split}
        \end{equation*}
        Thus, $\norm{A^{*,\alpha}T'(-s,-t)y}_{X^*} \leq CM e^{-\gamma(s-t)} \norm{A^{*,\alpha}y}_{X^*}$,
        where the constant $C$ is determined by the family $P(.)$ of projections.
    \end{enumerate}
  \end{proof}

  To sum it up, we have proved that \eqref{eq:140106-1522b} satisfies
  (CH). In comparison to \eqref{eq:140106-1522}, 
  the Morse indices $m^-$ and $m^+$ are obviously swapped. This is caused by
  the reversal of the time variable.

  Let the spaces $\X$, $\Y$, $\Z$ be defined as in the previous section,
  and let $\X'$, $\Y'$, $\Z'$ denote their dual counterparts, that is,
  \begin{align*}
    \X' &:= C^{1,\delta}_B(\IR,X^*) \cap C^{0,\delta}_B(\IR,X^{*,1}) \\
    \Z' &:= C^{0,\delta}_B(\IR,X^*).
  \end{align*}
  We consider the operators $L\in \mathcal{L}(\X,\Z)$
  (resp. $L'\in\mathcal{L}(\X',\Z')$) defined by
  \begin{equation*}
    Lu := u_t + Au - B(t)u
  \end{equation*}
  and
  \begin{equation*}
    L'v := v_t + A^*v - B'(t)v.
  \end{equation*}

  \begin{lemma}
    \label{le:140107-1702}
    If $\Ran(L)\supset \X$, then $\Null(L') = \{0\}$.
    Analogously, if $\Ran(L')\supset \X'$, then $\Null(L) = \{0\}$.
  \end{lemma}

  \begin{proof}
    Assume that $L'v = 0$ for some $v\in \X'$ and let
    $u\in\X$ satisfy $u(t)\to 0$ as $\abs{t}\to\infty$. Integration by parts
    shows that 
    \begin{equation*}
      \begin{split}
        &\int\limits^a_{-a} \spr{Lu(s), A^{*,\alpha}v(-s)}\,\de s \\
        &\quad= \int\limits^a_{-a} \spr{u_t(s), A^{*,\alpha}v(-s)} + \spr{A^\alpha u(s), A^*v(-s) - B'(-s)v(-s)}\,\de s\\
        &\quad= (u(a),v(-a)) - (u(-a),v(-a)) + \int\limits^a_{-a} \spr{A^\alpha u(s), \underbrace{v_t(-s) + (A^* - B'(-s))v(-s)}_{=(L'v)(-s) \equiv 0}} \,\de s.\\
      \end{split}
    \end{equation*}
    Consequently for all $u\in\X$ with $u(t)\to 0$ as $\abs{t}\to\infty$, one has
    \begin{equation}
      \label{eq:140106-1748}
      \int\limits^a_{-a} \spr{(Lu)(s), A^{*,\alpha}v(-s)}\,\de s \to 0\text{ as }a\to\infty.
    \end{equation}

    Arguing by contradiction, suppose that $v(t_0)\neq 0$ for some
    $t_0\in\IR$. Since \cite[Theorem 2.6.8]{pazy} $X^1$ is dense in $X$,
    there is an $x_0\in X^1$ such that $\spr{x_0,A^{*,\alpha}v(t_0)}\neq 0$.
    Choose $w\in C^{1,\delta}_{B}(\IR, X^1)$ such that $w(t_0) = x_0$ and $w(t) = 0$ for all
    $t\in\IR$ with $\abs{t-t_0}\geq \eps$. For small $\eps>0$, we have
    \begin{equation*}
      C := \int\limits^\infty_{-\infty} \spr{\underbrace{w(s)}_{\in \Ran(L)}, A^{*,\alpha}v(-s)} \,\de s \neq 0.
    \end{equation*}
    We further have $w\in \X \subset \Ran(L)$, that is, $w = Lu$ for some
    $u\in\X$. Since $w(t) = 0$ for $\abs{t}$ sufficiently large, it follows from (CH) respectively
    from the existence of exponential dichotomies at $\infty$ and $-\infty$ that
    $u(t)\to 0$ as $\abs{t}\to 0$. Hence, one has $C=0$ by \eqref{eq:140106-1748}, which is a contradiction.

    Using the Hahn-Banach theorem, the second claim can be treated similarly.
  \end{proof}
  
  \begin{lemma}
    \label{le:140107-1701}
    Suppose that $L$ is surjective. Then:
    \begin{enumerate}
      \item $m^-\geq m^+$;
      \item if $m^-=m^+$, then $L$ is also injective.
    \end{enumerate}
  \end{lemma}

  \begin{proof}
    \begin{enumerate}
      \item 
        Assume to the contrary that $m^-<m^+$.
        Let $P^-$ (resp. $P^+$) denote the projections
        associated with the exponential dichotomy at $-\infty$ (resp. $+\infty$),
        which are given by (CH). Let $(P^-)'$ and $(P^+)'$ defined
        by Lemma \ref{le:140409-1425}. Note that $\dim\Ran(P^-)' = m^-$
        and $\dim\Ran(P^+)' = m^+$.

        By $T'(t,s)$, we mean the evolution operator on $X^{\alpha,*}$ 
        defined by \eqref{eq:140106-1522b}. Let $t_1<0<t_2$ so that $(P^+)'(t_1)$ and
         $(P^-)'(t_2)$ are defined. Since $m^-<m^+$, the operator 
         $(P^-(t_2))'T'(t_2,t_1):\; \Ran(P^+)'(t_1)\to \Ran(P^-)'(t_2)$
		is not injective. Therefore, there exists a non-trivial bounded solution
        of \eqref{eq:140106-1522b}, in contradiction to Lemma \ref{le:140107-1702}.
      \item
        $L'$ is injective by Lemma \ref{le:140107-1702}.
        We can now apply Lemma \ref{le:131219-1505} to $L'$, 
        showing that $L'$ is also surjective. 
        Finally, Lemma \ref{le:140107-1702} implies that
        $L$ is injective as claimed.
    \end{enumerate}
  \end{proof}
\end{section}

\providecommand{\bysame}{\leavevmode\hbox to3em{\hrulefill}\thinspace}
\providecommand{\MR}{\relax\ifhmode\unskip\space\fi MR }
\providecommand{\MRhref}[2]{%
	\href{http://www.ams.org/mathscinet-getitem?mr=#1}{#2}
}
\providecommand{\href}[2]{#2}

\end{document}